\documentclass{amsart}
\usepackage{latexsym,amssymb,amsthm,amsmath}

\usepackage{amsthm}
\usepackage{amssymb}
\usepackage{amsmath}

\theoremstyle{plain}
\newtheorem{theorem}{Theorem}[section]
\newtheorem*{Theorem B}{Theorem B}
\newtheorem*{Theorem A}{Theorem A}
\newtheorem{lemma}{Lemma}[section]
\newtheorem{proposition}{Proposition}[section]

\numberwithin{equation}{section}

\theoremstyle{remark}
\newtheorem{remark}{Remark}[section]
\begin{document}
\title[Submanifolds of Generalized Sasakian-space-forms$\cdots$]{Submanifolds of Generalized Sasakian-space-forms with respect to certain connections}
\author[P. Mandal, S. Kishor and S. K. Hui]{Pradip Mandal, Shyam Kishor and Shyamal Kumar Hui$^{*}$}
\keywords{generalized Sasakian-space-forms, semisymmetric metric connection, semisymmetric non-metric connection, Schouten-van Kampen Connection, Tanaka-Webster connection.\\{* corresponding author}}

\begin{abstract}
The present paper deals with some results of submanifolds of generalized Sasakian-space-forms in \cite{ALEGRE3}
with respect to semisymmetric metric connection, semisymmetric non-metric connection,
Schouten-van Kampen connection and Tanaka-webster connection.
\end{abstract}
\maketitle
\section{Introduction}
As a generalization of Sasakian-space-form, Alegre et al. \cite{ALEGRE1} introduced the notion of generalized
Sasakian-space-form as that an almost contact metric manifold
$\bar{M}(\phi,\xi,\eta,g)$
whose curvature tensor $\bar{R}$ of $\bar{M}$ satisfies
\begin{align}
\label{eqn1.3}
\bar{R}(X,Y)Z &=f_1\big\{g(Y,Z)X-g(X,Z)Y\big\}+f_2\big\{g(X,\phi Z)\phi Y\\
\nonumber& - g(Y,\phi Z)\phi X + 2g(X,\phi Y)\phi Z\big\}+f_3\big\{\eta(X)\eta(Z)Y\\
\nonumber& - \eta(Y)\eta(Z)X+g(X,Z)\eta(Y)\xi - g(Y,Z)\eta(X)\xi\big\}
\end{align}
for all vector fields $X$, $Y$, $Z$ on $\bar{M}$ and $f_1,f_2,f_3$ are certain smooth functions on
 $\bar{M}$. Such a manifold of dimension
$(2n+1)$, $n>1$ (the condition $n>1$ is assumed throughout the paper),
is denoted by $\bar{M}^{2n+1}(f_1,f_2,f_3)$ \cite{ALEGRE1}. Many authors studied this space form with different aspects.
For this, we may refer (\cite{HUI1}, \cite{HUI2}, \cite{HUI3}, \cite{HUI4}, \cite{HUI5}, \cite{HUI6}, \cite{KISH} and \cite{HUI8}).
It reduces to Sasakian-space-form if $f_1 = \frac{c+3}{4}$, $f_2 = f_3 = \frac{c-1}{4}$ \cite{ALEGRE1}.

After introducing the semisymmetric linear connection by Friedman and Schouten \cite{FRID},
Hayden \cite{HAYD} gave the idea of metric connection with torsion on a Riemannian manifold.
Later, Yano \cite{YANO} and many others (see, \cite{HUI7}, \cite{SHAIKH1}, \cite{SULAR} and references therein)
studied semisymmetric metric connection in different context. The idea of semisymmetric non-metric
connection was introduced by Agashe and Chafle \cite{AGAS}.

The Schouten-van Kampen connection introduced
 for the study of non-holomorphic manifolds (\cite{SCHO}, \cite{VRAN}). In $2006$, Bejancu \cite{BEJA3}
studied Schouten-van Kampen connection on foliated manifolds. Recently Olszak \cite{OLSZ}
studied Schouten-van Kampen connection on almost(para) contact metric structure.

The Tanaka-Webster connection (\cite{TANA}, \cite{WEBS}) is the canonical affine connection defined on a non-degenerate
pseudo-Hermitian CR-manifold. Tanno \cite{TANN} defined the Tanaka-Webster connection for contact metric manifolds.

The submanifolds of $\bar{M}^{2n+1}(f_1,f_2,f_3)$ are studied in (\cite{ALEGRE3}, \cite{PM1}, \cite{PM2}).
In \cite{ALEGRE3}, Alegre and Carriazo studied submanifolds of $\bar{M}^{2n+1}(f_1,f_2,f_3)$ with respect
to Levi-Civita connection $\bar{\nabla}$. The present paper deals with study of such
submanifolds of $\bar{M}^{2n+1}(f_1,f_2,f_3)$ with respect to semisymmetric metric connection,
semisymmetric non-metric connection, Schouten-van Kampen connection and Tanaka-webster connection respectively.
\section{preliminaries}
In an almost contact metric manifold $\bar{M}(\phi,\xi,\eta,g)$, we have \cite{BLAIR}
\begin{align}
\label{eqn2.1}
\phi^2(X) = -X+\eta(X)\xi,\ \phi \xi=0,
\end{align}
\begin{align}
\label{eqn2.2}
\eta(\xi) = 1,\ g(X,\xi) = \eta(X),\ \eta(\phi X) = 0,
\end{align}
\begin{align}
\label{eqn2.3}
g(\phi X,\phi Y) = g(X,Y)-\eta(X)\eta(Y),
\end{align}
\begin{align}
\label{eqn2.4} g(\phi X,Y) = -g(X,\phi Y).
\end{align}
In
$\bar{M}^{2n+1}(f_1,f_2,f_3)$, we have \cite{ALEGRE1}
\begin{align}
\label{eqn2.6f}
(\bar{\nabla}_X\phi)(Y) = (f_1-f_3)[g(X,Y)\xi - \eta(Y)X],
\end{align}
\begin{align}
\label{eqn2.7g}
\bar{\nabla}_X\xi = -(f_1-f_3) \phi X,
\end{align}
where $\bar{\nabla}$ is the Levi-Civita connection of $\bar{M}^{2n+1}(f_1,f_2,f_3)$.

\indent Let $M$ be a submanifold of
$\bar{M}^{2n+1}(f_1,f_2,f_3)$. If $\nabla$ and $\nabla^\perp$ are the induced
connections on the tangent bundle $TM$ and the normal bundle $T^\perp{M}$ of $M$, respectively
then the Gauss and Weingarten formulae are given by \cite{YANOKON1}
\begin{align}
\label{eqn2.5}
\bar{\nabla}_XY = \nabla_XY +h(X,Y),\ \bar{\nabla}_XV = -A_VX + \nabla_X^{\perp}V
\end{align}
for all $X,Y\in\Gamma(TM)$ and $V\in\Gamma(T^{\perp}M)$, where $h$ and $A_V$ are second fundamental form and shape operator (corresponding to the normal vector field V), respectively and they are related by \cite{YANOKON1} $ g(h(X,Y),V) = g(A_VX,Y)$.

For any $X\in\Gamma(TM)$, we may write
\begin{equation}
\label{eqn2.17f} \phi X=TX+FX,
\end{equation}
where $TX$ is the tangential component and $FX$ is the
normal component of $\phi X$.

In particular, if $F=0$ then $M$ is invariant \cite{BEJA} and here $\phi (TM)\subset TM$.
Also if $T=0$ then $M$ is anti-invariant \cite{BEJA} and here $\phi (TM)\subset T^\bot M$.
Also here we assume that $\xi$ is tangent to $M$.

The semisymmetric metric connection $\widetilde{\bar{\nabla}}$ and the Riemannian connection $\bar{\nabla}$
on ${\bar{M}}^{2n+1}(f_{1},f_{2},f_{3})$ are related by \cite{YANO}
\begin{align}
\label{eqn2.41d}
 \widetilde{\bar{\nabla}}_{X}Y= \bar{\nabla}_X Y+\eta(Y)X-g(X,Y)\xi.
\end{align}
The Riemannian curvature tensor $\widetilde{\bar{R}}$ of  $\bar{M}^{2n+1}(f_{1},f_{2},f_{3})$ with respect to $\widetilde{\bar{\nabla}}$ is
\begin{eqnarray}
    \label{eqn1.3s}
&&\widetilde{\bar{R}}(X,Y)Z \\
\nonumber&&=(f_1-1)\big\{g(Y,Z)X-g(X,Z)Y\big\}+f_2\big\{g(X,\phi Z)\phi Y- g(Y,\phi Z)\phi X\\
\nonumber&&  + 2g(X,\phi Y)\phi Z\big\}+(f_3-1)\big\{\eta(X)\eta(Z)Y - \eta(Y)\eta(Z)X+g(X,Z)\eta(Y)\xi\\
\nonumber&& - g(Y,Z)\eta(X)\xi\big\}+(f_1-f_3)\{g( X, \phi Z)Y-g( Y,\phi Z)X+g(Y,Z)\phi X\\
\nonumber&&-g(X,Z)\phi Y\}.
\end{eqnarray}
The semisymmetric non-metric connection ${\bar{\nabla}}^{'}$ and the Riemannian connection $\bar{\nabla}$
on ${\bar{M}}^{2n+1}(f_{1},f_{2},f_{3})$ are related by \cite{AGAS}
\begin{align}
\label{eqn2.41h}
\bar{\nabla}^{'}_{X}Y= \bar{\nabla}_X Y+\eta(Y)X.
\end{align}
The Riemannian curvature tensor ${\bar{R}}^{'}$ of  $\bar{M}^{2n+1}(f_{1},f_{2},f_{3})$ with respect to ${\bar{\nabla}}^{'}$ is
\begin{eqnarray}
\label{eqn1.3t}
{\bar{R}}^{'}(X,Y)Z &=&f_1\big\{g(Y,Z)X-g(X,Z)Y\big\}+f_2\big\{g(X,\phi Z)\phi Y\\
\nonumber&-&g(Y,\phi Z)\phi X + 2g(X,\phi Y)\phi Z\big\}+f_3\big\{\eta(X)\eta(Z)Y\\
\nonumber& -&\eta(Y)\eta(Z)X+g(X,Z)\eta(Y)\xi - g(Y,Z)\eta(X)\xi\big\}\\
\nonumber&+&(f_1-f_3)[g(X,\phi  Z)Y-g( Y,\phi Z) X]\\
\nonumber&+&\eta(Y)\eta(Z)X-\eta(X)\eta(Z)Y.
\end{eqnarray}
The Schouten-van Kampen connection $\hat{\bar{\nabla}}$ and the Riemannian connection $\bar{\nabla}$
of ${\bar{M}}^{2n+1}(f_{1},f_{2},f_{3})$ are related by \cite{OLSZ}
\begin{align}
\label{eqn2.41k}
\hat{\bar{\nabla}}_{X}Y=\bar{\nabla}_X Y+(f_1-f_3)\eta(Y)\phi X-(f_1-f_3)g(\phi X,Y)\xi.
\end{align}
The Riemannian curvature tensor $\hat{\bar{R}}$ of
$\bar{M}^{2n+1}(f_{1},f_{2},f_{3})$ with respect to $\hat{\bar{\nabla}}$ is
\begin{eqnarray}
\label{eqn1.3v}
&&\hat{\bar{R}}(X,Y)Z \\
\nonumber&&=f_1\big\{g(Y,Z)X-g(X,Z)Y\big\}+f_2\big\{g(X,\phi Z)\phi Y\\
\nonumber&&- g(Y,\phi Z)\phi X+ 2g(X,\phi Y)\phi Z\big\}+\{f_3+(f_1-f_3)^2\}\big\{\eta(X)\eta(Z)Y \\
\nonumber&&- \eta(Y)\eta(Z)X+g(X,Z)\eta(Y)\xi - g(Y,Z)\eta(X)\xi\big\}\\
\nonumber&&+(f_1-f_3)^2\big[g(X,\phi Z)\phi Y-g(Y, \phi Z)\phi X\big].
\end{eqnarray}
The Tanaka-Webster connection $\stackrel{\ast}{\bar{\nabla}}$ and the Riemannian connection $\bar{\nabla}$
of ${\bar{M}}^{2n+1}(f_{1},f_{2},f_{3})$ are related by \cite{CHO}
\begin{align}
\label{eqn2.41e}
 \stackrel{\ast}{\bar{\nabla}}_{X}Y= \bar{\nabla}_X Y+\eta(X)\phi Y+(f_1-f_3)\eta(Y)\phi X-(f_1-f_3)g(\phi X,Y)\xi.
\end{align}
The Riemannian curvature tensor $\stackrel{*}{\bar{R}}$ of
$\bar{M}^{2n+1}(f_{1},f_{2},f_{3})$ with respect to $\stackrel{*}{\bar{\nabla}}$ is
\begin{eqnarray}
\label{eqn1.3u}
&&\stackrel{*}{\bar{R}}(X,Y)Z \\
\nonumber&&=f_1\big\{g(Y,Z)X-g(X,Z)Y\big\}+f_2\big\{g(X,\phi Z)\phi Y- g(Y,\phi Z)\phi X\\
\nonumber&&+ 2g(X,\phi Y)\phi Z\big\}+\{f_3+{(f_1-f_3)^2}\}\big\{\eta(X)\eta(Z)Y - \eta(Y)\eta(Z)X\\
\nonumber&&+g(X,Z)\eta(Y)\xi - g(Y,Z)\eta(X)\xi\big\}+{(f_1-f_3)^2}\big[g(X,\phi Z)\phi Y\\
\nonumber&&-g(Y,\phi Z)\phi X\big]+2(f_1-f_3)g(X,\phi Y)\phi Z.
\end{eqnarray}
\section{Submanifolds of $\bar{M}^{2n+1}(f_1,f_2,f_3)$ with $\widetilde{\bar{\nabla}}$}
\begin{lemma}
If $M$ is invariant submanifold of $\bar{M}^{2n+1}(f_1,f_2,f_3)$ with respect to $\widetilde{\bar{\nabla}}$, then
$\widetilde{\bar{R}}(X,Y)Z$ is tangent to $M$, for any $X,Y,Z\in \Gamma(TM)$.
\end{lemma}
\begin{proof}
If $M$ is invariant then from (\ref{eqn1.3s}) we say that $\widetilde{\bar{R}}(X,Y)Z$ is
tangent to $M$ because $\phi X$ and $\phi Y$ are tangent to $M$. This proves the lemma.
\end{proof}
\begin{lemma}
If $M$ is anti-invariant submanifold of $\bar{M}^{2n+1}(f_1,f_2,f_3)$ with respect to $\widetilde{\bar{\nabla}}$, then
\begin{eqnarray}
\label{eqn7.3}
 &&tan(\widetilde{\bar{R}}(X,Y)Z) \\
\nonumber&&=(f_1-1)\big\{g(Y,Z)X-g(X,Z)Y\big\}+(f_3-1)\big\{\eta(X)\eta(Z)Y\\
\nonumber&&-\eta(Y)\eta(Z)X+g(X,Z)\eta(Y)\xi-g(Y,Z)\eta(X)\xi\big\},
\end{eqnarray}
\begin{eqnarray}
\label{eqn7.4}
 nor(\widetilde{\bar{R}}(X,Y)Z)&=&(f_1-f_3)\{g(Y,Z)\phi X-g(X,Z)\phi Y\}
\end{eqnarray}
for any $X,Y,Z\in \Gamma(TM)$.
\end{lemma}
\begin{proof}
Since $M$ is anti-invariant, we have $\phi X,\phi Y\in \Gamma(T^\bot M)$.
Then equating tangent and normal component of (\ref{eqn1.3s}) we get the result.
\end{proof}
\begin{lemma}
If $f_1(p)=f_3(p)$ and $M$ is either invariant or anti-invariant submanifold
of $\bar{M}^{2n+1}(f_1,f_2,f_3)$ with respect to $\widetilde{\bar{\nabla}}$, then
$\widetilde{\bar{R}}(X,Y)Z$ is tangent to $M$ for any $X,Y,Z\in \Gamma(TM)$.
\end{lemma}
\begin{proof}
Using Lemma $3.1$ and Lemma $3.2$ we get the result.
\end{proof}
\begin{lemma}
If $M$ is invariant or anti-invariant submanifold of $\bar{M}^{2n+1}(f_1,f_2,f_3)$ with respect to $\widetilde{\bar{\nabla}}$, then
$\widetilde{\bar{R}}(X,Y)V$ is normal to $M$, for any $X,Y,\in \Gamma(TM)$ and $V\in \Gamma(T^\bot M)$.
\end{lemma}
\begin{proof}
 If $M$ is invariant from (\ref{eqn1.3s}) we have $\widetilde{\bar{R}}(X,Y)V$ normal to $M$, and if $M$ is anti-invariant then $\widetilde{\bar{R}}(X,Y)V=0$ i.e. $\widetilde{\bar{R}}(X,Y)V$ normal to $M$ for any $X,Y,\in \Gamma(TM)$ and $V\in \Gamma(T^\bot M)$.
\end{proof}
\begin{lemma}
let $M$ be a connected submanifold of $\bar{M}^{2n+1}(f_1,f_2,f_3)$ with
respect to $\widetilde{\bar{\nabla}}$. If $f_2(p)\neq0$, $f_1(p)=f_3(p)$ and $TM$ is invariant
under the action of $\widetilde{\bar{R}}(X,Y)$, $X,Y\in \Gamma(TM)$, then $M$ is either invariant or anti-invariant.
\end{lemma}
\begin{proof}
For $X,Y\in \Gamma(TM)$, we have from  (\ref{eqn1.3s}) that
 \begin{eqnarray}
\label{eqn3.1}
\widetilde{\bar{R}}(X,Y)X&=&(f_1-1)\big\{g(Y,X)X-g(X,X)Y\big\}+f_2\big\{g(X,\phi X)\phi Y\\
\nonumber&-&g(Y,\phi X)\phi X + 2g(X,\phi Y)\phi X\big\}+(f_3-1)\big\{\eta(X)\eta(X)Y \\
\nonumber&-&\eta(Y)\eta(X)X+g(X,X)\eta(Y)\xi - g(Y,X)\eta(X)\xi\big\}\\
\nonumber&+&(f_1-f_3)\{g(\phi Y,X)X-g(\phi X,X)Y+g(Y,X)\phi X\\
\nonumber&-&g(X,X)\phi Y\}.
\end{eqnarray}
Note that $\widetilde{\bar{R}}(X,Y)X$ should be tangent if
$[-3f_2g(Y,\phi X)\phi X+(f_1-f_3)\{g(Y,X)\phi X-g(X,X)\phi Y\}]$ is tangent.
Since $f_2(p)\neq0$, $f_1(p)=f_3(p)$ at any point $p$ then by similar way of proof of Lemma $3.2$ of \cite{ALEGRE3}, we can prove that either $M$ is invariant or anti-invariant. This proves the Lemma.
\end{proof}
\begin{remark}
let $M$ be a connected submanifold of $\bar{M}^{2n+1}(f_1,f_2,f_3)$ with
respect to $\widetilde{\bar{\nabla}}$. If $f_1(p)\neq f_3(p)$ and $TM$ is invariant
under the action of $\widetilde{\bar{R}}(X,Y)$, $X,Y\in \Gamma(TM)$, then $M$ is invariant.
\end{remark}
From Lemma $3.3$ and Lemma $3.5$, we have
\begin{theorem}
Let $M$ be a connected submanifold of $\bar{M}^{2n+1}(f_1,f_2,f_3)$ with respect
to  $\widetilde{\bar{\nabla}}$. If $f_2(p)\neq0$, $f_1(p)=f_3(p)$ then $M$ is either invariant or anti-invariant if and only if $TM$ is
invariant under the action of $\widetilde{\bar{R}}(X,Y)$ for all $X,Y\in\Gamma(TM)$.
\end{theorem}
\begin{proposition}
Let $M$ be a submanifold of $\bar{M}^{2n+1}(f_1,f_2,f_3)$ with respect to
$\widetilde{\bar{\nabla}}$. If $M$ is invariant, then $TM$ is invariant under
the action of $\widetilde{\bar{R}}(U,V)$ for any $U,V\in \Gamma(T^\bot M)$.
\end{proposition}
\begin{proof}
Replacing $X,Y,Z$ by $U,V,X$ in (\ref{eqn1.3s}), we get
\begin{eqnarray}
\label{eqn3.2}
\widetilde{\bar{R}}(U,V)X &=&(f_1-1)\big\{g(V,X)U-g(U,X)V\big\}+f_2\big\{g(U,\phi X)\phi V\\
\nonumber&-& g(V,\phi X)\phi U + 2g(U,\phi V)\phi X\big\}+(f_3-1)\big\{\eta(U)\eta(X)V\\
\nonumber&-& \eta(V)\eta(X)U+g(U,X)\eta(V)\xi - g(V,X)\eta(U)\xi\big\}\\
\nonumber&+&(f_1-f_3)\{g(\phi V,X)U-g(\phi U,X)V+g(V,X)\phi U \\
\nonumber&-&g(U,X)\phi V\}.
\end{eqnarray}
As $M$ is invariant, $U,V\in \Gamma(T^\bot M)$, we have
\begin{equation}\label{eqn3.3}
g(X,\phi U)=-g(\phi X,U)=g(\phi V,X)=0
\end{equation}
 for any $X\in \Gamma(TM)$. Using (\ref{eqn3.3}) in (\ref{eqn3.2}), we  have
 \begin{equation}\label{eqn3.4}
 \widetilde{\bar{R}}(U,V)X=2f_2g(U,\phi V)\phi X,
 \end{equation}
 which is tangent as $\phi X$ is tangent.
This proves the proposition.
\end{proof}
\begin{proposition}
Let $M$ be a connected submanifold of $\bar{M}^{2n+1}(f_1,f_2,f_3)$ with respect
to $\widetilde{\bar{\nabla}}$. If $f_2(p)\neq0$, $f_1(p)=f_3(p)$ for each $p\in M$ and $T^\bot M$ is invariant under the action
of $\widetilde{\bar{R}}(U,V)$, $U,V\in\Gamma(T^\bot M)$, then $M$ is either invariant or anti-invariant.
\end{proposition}
\begin{proof}
The proof is similar as it is an Lemma $3.4$, just assuming that $\widetilde{\bar{R}}(U,V)U$ is normal for any $U,V\in \Gamma(T^\bot M)$.
\end{proof}
\section{Submanifolds of $\bar{M}^{2n+1}(f_1,f_2,f_3)$ with ${\bar{\nabla}}^{'}$}
\begin{lemma}
If $M$ is either invariant or anti-invarint submanifold of $\bar{M}^{2n+1}(f_1,f_2,f_3)$ with
respect to ${\bar{\nabla}}^{'}$, then
${\bar{R}}^{'}(X,Y)Z$ is tangent to $M$ and ${\bar{R}}^{'}(X,Y)V$ normal to $M$ for any $X,Y,Z\in \Gamma(TM)$ and $V\in \Gamma(T^\bot M)$.
\end{lemma}
\begin{proof}
If $M$ is invariant then from (\ref{eqn1.3t}) we say that ${\bar{R}}^{'}(X,Y)Z$ is
tangent to $M$ because $\phi X$ and $\phi Y$ are tangent to $M$. \\
If $M$ is anti-invariant then
\begin{equation}\label{eqn4.1}
g(X,\phi Z)=g(Y,\phi Z)=g(\phi X,Z)=g(\phi Y,Z)=0.
\end{equation}
From (\ref{eqn1.3t}) and (\ref{eqn4.1}) we have
\begin{eqnarray}
 \label{eqn4.2}
{\bar{R}}^{'}(X,Y)Z &=&f_1\big\{g(Y,Z)X-g(X,Z)Y\big\}+f_3\big\{\eta(X)\eta(Z)Y \\
\nonumber&-& \eta(Y)\eta(Z)X+g(X,Z)\eta(Y)\xi - g(Y,Z)\eta(X)\xi\big\}\\
\nonumber&+&[\eta(Y)\eta(Z)X-\eta(X)\eta(Z)Y],
\end{eqnarray}
which is tangent. \\
If $M$ is invariant then from (\ref{eqn1.3t}), it follows that ${\bar{R}}^{'}(X,Y)V$ is normal to $M$, and if $M$ is anti-invariant then ${\bar{R}}^{'}(X,Y)V=0$ i.e. ${\bar{R}}^{'}(X,Y)V$ is normal to $M$ for any $X,Y\in \Gamma(TM)$ and $V\in \Gamma(T^\bot M)$.
This proves the Lemma.
\end{proof}
\begin{lemma}
Let $M$ be a connected submanifold of $\bar{M}^{2n+1}(f_1,f_2,f_3)$ with respect to ${\bar{\nabla}}^{'}$. If $f_2(p)\neq 0$ for each $p\in M$ and $TM$ is invariant
under the action of $\bar{R}^{'}(X,Y)$, $X,Y\in \Gamma(TM)$,
then $M$ is either invariant or anti-invariant.
\end{lemma}
\begin{proof}
For $X,Y\in \Gamma(TM)$, we have from (\ref{eqn1.3t}) that
\begin{eqnarray}
\label{eqn4.3}
  {\bar{R}}^{'}(X,Y)X &=&f_1\big\{g(Y,X)X-g(X,X)Y\big\}+f_2\big\{g(X,\phi X)\phi Y\\
\nonumber&-&g(Y,\phi X)\phi X + 2g(X,\phi Y)\phi X\big\}+f_3\big\{\eta(X)\eta(X)Y\\
\nonumber&-&\eta(Y)\eta(X)X+g(X,X)\eta(Y)\xi - g(Y,X)\eta(X)\xi\big\}\\
\nonumber&-&(f_1-f_3)g(\phi X, Y) X+\{\eta(Y)\eta(Z)X-\eta(X)\eta(Z)Y\}.
\end{eqnarray}
Note that  ${\bar{R}}^{'}(X,Y)X$ should be tangent if $3f_2(p)g(Y,\phi X)\phi X$ is tangent.
Since $f_2(p)\neq0$ for each $p\in M$, as similar as proof of Lemma $3.2$ of \cite{ALEGRE3},
we may conclude that either $M$ is invariant or anti-invariant. This proves the Lemma.
\end{proof}
From Lemma $4.1$ and Lemma $4.2$, we have
\begin{theorem}
Let $M$ be a connected submanifold of $\bar{M}^{2n+1}(f_1,f_2,f_3)$ with respect
to ${\bar{\nabla}}^{'}$. If $f_2(p)\neq0$ for each $p\in M$, then $M$ is
either invariant or anti-invariant if and only if $TM$ is invariant under the
action of ${\bar{R}}^{'}(X,Y)$ for all $X,Y\in\Gamma(TM)$.
\end{theorem}
\begin{proposition}
Let $M$ be a submanifold of $\bar{M}^{2n+1}(f_1,f_2,f_3)$ with respect to
${\bar{\nabla}}^{'}$. If $M$ is invariant, then $TM$ is invariant
under the action of ${\bar{R}}^{'}(U,V)$ for any $U,V\in \Gamma(T^\bot M)$.
\end{proposition}
\begin{proof}
Replacing $X,Y,Z$ by $U,V,X$ in (\ref{eqn1.3t}), we get
\begin{eqnarray}
\label{eqn4.4}
{\bar{R}}^{'}(U,V)X &=&f_1\big\{g(V,X)U-g(U,X)V\big\} +f_2\big\{g(U,\phi X)\phi V\\
\nonumber&-& g(V,\phi X)\phi U + 2g(U,\phi V)\phi X\big\}+f_3\big\{\eta(U)\eta(X)V\\
\nonumber&-& \eta(V)\eta(X)U+g(U,X)\eta(V)\xi - g(V,X)\eta(U)\xi\big\}\\
\nonumber&+&(f_1-f_3)\{g( U,\phi X) V-g( V,\phi X)U\}\\
\nonumber&+&\{\eta(V)\eta(X)U-\eta(U)\eta(X)V\}.
\end{eqnarray}
As $M$ is invariant, $U\in \Gamma(T^\bot M)$, we have
\begin{equation}\label{eqn4.5}
g(X,\phi U)=-g(\phi X,U)=g(\phi V,X)=0
\end{equation}
 for any $X\in \Gamma(TM)$. Using (\ref{eqn4.5}) in (\ref{eqn4.4}), we  have
 \begin{equation}\label{eqn4.6}
 {\bar{R}}^{'}(U,V)X=2f_2g(U,\phi V)\phi X,
 \end{equation}
 which is tangent as $\phi X$ is tangent. This proves the proposition.
\end{proof}
\begin{proposition}
Let $M$ be a connected submanifold of $\bar{M}^{2n+1}(f_1,f_2,f_3)$ with respect to
${\bar{\nabla}}^{'}$. If $f_2(p)\neq0$ for each $p\in M$ and
$T^\bot M$ is invariant under the action of ${\bar{R}}(U,V)$, $U,V\in\Gamma(TM)$,
then $M$ is either invariant or anti-invariant.
\end{proposition}
\begin{proof}
  The proof is similar as the proof of Lemma $4.2$, just imposing that ${\bar{R}}^{'}(U,V)U$ is normal for any $U,V\in \Gamma(TM)$.
\end{proof}
\section{Submanifolds of $\bar{M}^{2n+1}(f_1,f_2,f_3)$ with $\hat{\bar{\nabla}}$}
\begin{lemma}
If $M$ is either invariant or anti-invarint submanifold of $\bar{M}^{2n+1}(f_1,f_2,f_3)$ with
respect to $\hat{\bar{\nabla}}$, then
$\hat{\bar{R}}(X,Y)Z$ is tangent to $M$ and $\hat{\bar{R}}(X,Y)V$ is
normal to $M$ for any $X,Y,Z\in \Gamma(TM)$ and $V\in \Gamma(T^\bot M)$.
\end{lemma}
\begin{proof}
If $M$ is invariant then from (\ref{eqn1.3v}) we say that $\hat{\bar{R}}(X,Y)Z$ is
tangent to $M$ because $\phi X$ and $\phi Y$ are tangent to $M$.

\noindent If $M$ is anti-invariant then
\begin{equation}\label{eqn6.5.1}
g(X,\phi Z)=g(Y,\phi Z)=g(\phi X,Z)=g(\phi Y,Z)=0.
\end{equation}
From (\ref{eqn1.3v}) and (\ref{eqn6.5.1}) we have
\begin{eqnarray}
 \label{eqn6.5.2}
\hat{\bar{R}}(X,Y)Z &=&f_1\big\{g(Y,Z)X-g(X,Z)Y\big\}\\
\nonumber&+&\{f_3+(f_1-f_3)^2\}\big\{\eta(X)\eta(Z)Y-\eta(Y)\eta(Z)X \\
\nonumber&+& g(X,Z)\eta(Y)\xi - g(Y,Z)\eta(X)\xi\big\},
\end{eqnarray}
which is tangent.\\
If $M$ is invariant from (\ref{eqn1.3v}) we have $\hat{\bar{R}}(X,Y)V$ is normal to $M$, and if $M$ is anti-invariant then $\hat{\bar{R}}(X,Y)V=0$ i.e. $\hat{\bar{R}}(X,Y)V$ is normal to $M$ for any $X,Y\in \Gamma(TM)$ and $V\in \Gamma(T^\bot M)$.
 This proves the Lemma.
\end{proof}
\begin{lemma}
let $M$ be a connected submanifold of $\bar{M}^{2n+1}(f_1,f_2,f_3)$ with respect to $\hat{\bar{\nabla}}$.
If $3f_2\neq(f_1-f_3)^2$ on $M$ and $TM$ is invariant
under the action of $\hat{\bar{R}}(X,Y)$, $X,Y\in \Gamma(TM)$,
then $M$ is either invariant or anti-invariant.
\end{lemma}
\begin{proof}
For $X,Y\in \Gamma(TM)$, we have from (\ref{eqn1.3v}) that
\begin{eqnarray}
\label{eqn6.5.3}
  \hat{\bar{R}}(X,Y)X &=&f_1\big\{g(Y,X)X-g(X,X)Y\big\}+f_2\big\{g(X,\phi X)\phi Y\\
\nonumber&-& g(Y,\phi X)\phi X + 2g(X,\phi Y)\phi X\big\}\\
\nonumber&+&\{f_3+(f_1-f_3)^2\}\big\{\eta(X)\eta(X)Y-\eta(Y)\eta(X)X\\
\nonumber&+& g(X,X)\eta(Y)\xi-g(Y,X)\eta(X)\xi\big\}\\
\nonumber&+&(f_1-f_3)^2\big\{g(X,\phi X)\phi Y- g(Y,\phi X)\phi X\big\}.
\end{eqnarray}
Now, we see that  $\hat{\bar{R}}(X,Y)X$ should be tangent if $\{3f_2+(f_1-f_3)^2\}g(Y,\phi X)\phi X$ is tangent.
Since $3f_2\neq-(f_1-f_3)^2$ then in similar way of proof of Lemma $3.2$ of \cite{ALEGRE3} we may conclude that either $M$ is invariant or anti-invariant. This proves the Lemma.
\end{proof}
From Lemma $5.1$ and Lemma $5.2$, we can state the following:
\begin{theorem}
Let $M$ be a connected submanifold of $\bar{M}^{2n+1}(f_1,f_2,f_3)$ with respect
to $\hat{\bar{\nabla}}$. If $3f_2\neq-(f_1-f_3)^2$, then $M$ is
either invariant or anti-invariant if and only if $TM$ is invariant under the
action of $\hat{\bar{R}}(X,Y)$ for all $X,Y\in\Gamma(TM)$.
\end{theorem}
\begin{proposition}
Let $M$ be a submanifold of $\bar{M}^{2n+1}(f_1,f_2,f_3)$ with respect to
$\hat{\bar{\nabla}}$. If $M$ is invariant, then $TM$ is invariant
under the action of $\hat{\bar{R}}(U,V)$ for any $U,V\in \Gamma(T^\bot M)$.
\end{proposition}
\begin{proof}
Replacing $X,Y,Z$ by $U,V,X$ in (\ref{eqn1.3v}), we get
\begin{eqnarray}
\label{eqn6.5.4}
\hat{\bar{R}}(U,V)X &=&f_1\big\{g(V,X)U-g(U,X)V\big\} +f_2\big\{g(U,\phi X)\phi V\\
\nonumber&-&g(V,\phi X)\phi U + 2g(U,\phi V)\phi X\big\}\\
\nonumber&+&\{f_3+(f_1-f_3)^2\}\big\{\eta(U)\eta(X)V-\eta(V)\eta(X)U\\
\nonumber&+&g(U,X)\eta(V)\xi - g(V,X)\eta(U)\xi\big\}\\
\nonumber&+&(f_1-f_3)^2\big\{g(U,\phi X)\phi V- g(V,\phi X)\phi U\big\}.
\end{eqnarray}
As $M$ is invariant, $U\in \Gamma(T^\bot M)$, we have
\begin{equation}\label{eqn6.5.5}
g(X,\phi U)=-g(\phi X,U)=g(\phi V,X)=0
\end{equation}
 for any $X\in \Gamma(TM)$. Using (\ref{eqn6.5.5}) in (\ref{eqn6.5.4}), we  have
 \begin{equation}\label{eqn6.5.6}
 \hat{\bar{R}}(U,V)X=2f_2g(U,\phi V)\phi X,
 \end{equation}
 which is tangent as $\phi X$ is tangent.
This proves the proposition.
\end{proof}
\begin{proposition}
Let $M$ be a connected submanifold of $\bar{M}^{2n+1}(f_1,f_2,f_3)$ with respect to
$\hat{\bar{\nabla}}$. If $3f_2\neq-(f_1-f_3)^2$ on $ M$ and
$T^\bot M$ is invariant under the action of $\hat{\bar{R}}(U,V)$, $U,V\in\Gamma(T^\bot M)$,
then $M$ is either invariant or anti-invariant.
\end{proposition}
\begin{proof}
  The proof is similar as the proof of Lemma $5.2$, just imposing that $\hat{\bar{R}}(U,V)U$ is normal for any $U,V\in \Gamma(T^\bot M)$.
\end{proof}
\section{Submanifolds of $\bar{M}^{2n+1}(f_1,f_2,f_3)$ with $\stackrel{*}{\bar{\nabla}}$}
\begin{lemma}
If $M$ is either invariant or anti-invarint submanifold of $\bar{M}^{2n+1}(f_1,f_2,f_3)$ with
respect to $\stackrel{*}{\bar{\nabla}}$, then
$\stackrel{*}{\bar{R}}(X,Y)Z$ is tangent to $M$ and $\stackrel{*}{\bar{R}}(X,Y)V$ is
normal to $M$ for any $X,Y,Z\in \Gamma(TM)$ and $V\in \Gamma(T^\bot M)$.
\end{lemma}
\begin{proof}
If $M$ is invariant then from (\ref{eqn1.3u}) we say that $\stackrel{*}{\bar{R}}(X,Y)Z$ is
tangent to $M$ because $\phi X$ and $\phi Y$ are tangent to $M$.

\noindent If $M$ is anti-invariant then
\begin{equation}\label{eqn5.1}
g(X,\phi Z)=g(Y,\phi Z)=g(\phi X,Z)=g(\phi Y,Z)=0.
\end{equation}
From (\ref{eqn1.3u}) and (\ref{eqn5.1}) we have
\begin{eqnarray}
 \label{eqn5.2}
\stackrel{*}{\bar{R}}(X,Y)Z &=&f_1\big\{g(Y,Z)X-g(X,Z)Y\big\}\\
\nonumber&+&\{f_3+(f_1-f_3)^2\}\big\{\eta(X)\eta(Z)Y- \eta(Y)\eta(Z)X\\
\nonumber&+&g(X,Z)\eta(Y)\xi - g(Y,Z)\eta(X)\xi\big\}
\end{eqnarray}
which is tangent. \\
If $M$ is invariant from (\ref{eqn1.3u}) we have $\stackrel{*}{\bar{R}}(X,Y)V$ normal to $M$ and if $M$ is anti-invariant then $\stackrel{*}{\bar{R}}(X,Y)V=0$ i.e. $\stackrel{*}{\bar{R}}(X,Y)V$ normal to $M$ for any $X,Y\in \Gamma(TM)$ and $V\in \Gamma(T^\bot M)$.
This proves the Lemma.
\end{proof}
\begin{lemma}
let $M$ be a connected submanifold of $\bar{M}^{2n+1}(f_1,f_2,f_3)$ with respect to $\stackrel{*}{\bar{\nabla}}$.
If $\{3f_2+2(f_1-f_3)+(f_1-f_3)^2\}(p)\neq0$ for each $p\in M$ and $TM$ is invariant
under the action of $\stackrel{*}{\bar{R}}(X,Y)$, $X,Y\in \Gamma(TM)$,
then $M$ is either invariant or anti-invariant.
\end{lemma}
\begin{proof}
For $X,Y\in \Gamma(TM)$, we have from (\ref{eqn1.3u}) that
\begin{eqnarray}
\label{eqn5.3}
  \stackrel{*}{\bar{R}}(X,Y)X&=&f_1\big\{g(Y,X)X-g(X,X)Y\big\}+f_2\big\{g(X,\phi X)\phi Y\\
\nonumber&-& g(Y,\phi X)\phi X + 2g(X,\phi Y)\phi X\big\}\\
\nonumber&+&\{f_3+(f_1-f_3)^2\}\big\{\eta(X)\eta(X)Y-\eta(Y)\eta(X)X \\
\nonumber&+&g(X,X)\eta(Y)\xi- g(Y,X)\eta(X)\xi\big\}\\
\nonumber&+&(f_1-f_3)^2\big\{g(X,\phi X)\phi Y- g(Y,\phi X)\phi X\big\}\\
\nonumber&+&2(f_1-f_3)g( X, \phi Y)\phi X.
\end{eqnarray}
Now we see that  $\stackrel{*}{\bar{R}}(X,Y)X$ should be tangent if $\{3f_2+2(f_1-f_3)+(f_1-f_3)^2\}(p)g(Y,\phi X)\phi X$ is tangent.
Since $\{3f_2+2(f_1-f_3)+(f_1-f_3)^2\}(p)\neq0$ then by similar way of
proof of Lemma $3.2$ of \cite{ALEGRE3} we can proved that either $M$
is invariant or anti-invariant. This proves the Lemma.
\end{proof}
From Lemma $6.1$ and Lemma $6.2$, we have
\begin{theorem}
Let $M$ be a connected submanifold of $\bar{M}^{2n+1}(f_1,f_2,f_3)$ with respect
to $\stackrel{*}{\bar{\nabla}}$. If $\{3f_2+2(f_1-f_3)+(f_1-f_3)^2\}(p)\neq0$, then $M$ is
either invariant or anti-invariant if and only if $TM$ is invariant under the
action of $\stackrel{*}{\bar{R}}(X,Y)$ for all $X,Y\in\Gamma(TM)$.
\end{theorem}
\begin{proposition}
Let $M$ be a submanifold of $\bar{M}^{2n+1}(f_1,f_2,f_3)$ with respect to
$\stackrel{*}{\bar{\nabla}}$. If $M$ is invariant, then $TM$ is invariant
under the action of $\stackrel{*}{\bar{R}}(U,V)$ for any $U,V\in \Gamma(T^\bot M)$.
\end{proposition}
\begin{proof}
Replacing $X,Y,Z$ by $U,V,X$ in (\ref{eqn1.3u}), we get
\begin{eqnarray}
\label{eqn5.4}
\stackrel{*}{\bar{R}}(U,V)X &=&f_1\big\{g(V,X)U-g(U,X)V\big\}+f_2\big\{g(U,\phi X)\phi V\\
\nonumber&-&g(V,\phi X)\phi U + 2g(U,\phi V)\phi X\big\}\\
\nonumber&+&\{f_3+(f_1-f_3)^2\}\big\{\eta(U)\eta(X)V-\eta(V)\eta(X)U\\
\nonumber&+&g(U,X)\eta(V)\xi- g(V,X)\eta(U)\xi\big\}\\
\nonumber&+&(f_1-f_3)^2\big\{g(U,\phi X)\phi V- g(V,\phi X)\phi U \big\}\\
\nonumber&+&2(f_1-f_3)g(U,\phi V)\phi X.
\end{eqnarray}
As $M$ is invariant, $U\in \Gamma(T^\bot M)$, we have
\begin{equation}\label{eqn5.5}
g(X,\phi U)=-g(\phi X,U)=g(\phi V,X)=0
\end{equation}
 for any $X\in \Gamma(TM)$. Using (\ref{eqn5.5}) in (\ref{eqn5.4}), we  have
 \begin{equation}\label{eqn5.6}
 \stackrel{*}{\bar{R}}(U,V)X=\{2f_2+2(f_1-f_3)\}g(U,\phi V)\phi X,
 \end{equation}
 which is tangent as $\phi X$ is tangent.
This proves the proposition.
\end{proof}
\begin{proposition}
Let $M$ be a connected submanifold of $\bar{M}^{2n+1}(f_1,f_2,f_3)$ with respect to
$\stackrel{*}{\bar{\nabla}}$. If $\{3f_2+2(f_1-f_3)+(f_1-f_3)^2\}(p)\neq0$ for each $p\in M$ and
$T^\bot M$ is invariant under the action of $\stackrel{*}{\bar{R}}(U,V)$, $U,V\in\Gamma(T^\bot M)$,
then $M$ is either invariant or anti-invariant.
\end{proposition}
\begin{proof}
The proof is similar as the proof of Lemma $6.2$, just considering
  that $\stackrel{*}{\bar{R}}(U,V)U$ is normal for any $U,V\in \Gamma(T^\bot M)$.
\end{proof}
\noindent{\bf Acknowledgement:} The first author (P. Mandal)  gratefully acknowledges to
the CSIR(File No.:09/025(0221)/2017-EMR-I), Govt. of India for financial assistance.
The Third author (S. K. Hui) are thankful to University of Burdwan for providing administrative and technical support.

\vspace{0.1in}

\noindent Pradip Mandal and Shyamal Kumar Hui\\
Department of Mathematics, The University of Burdwan, Burdwan -- 713104, West Bengal, India\\
E-mail: pm2621994@gmail.com; skhui@math.buruniv.ac.in\\

\vspace{0.1in}
\noindent Shyam Kishor\\
Department of Mathematics and Astronomy, University of Lucknow, Lucknow -- 226007, India\\
E-mail: skishormath@gmail.com
\end{document}